\documentclass[12pt,onecolumn,twoside]{IEEEtran}
\usepackage{amsmath,amssymb,euscript,yfonts,psfrag,latexsym,dsfont,graphicx,bbm,color,amstext,wasysym,pdfsync,graphicx}

\newtheorem{thm}{Theorem}
\newtheorem{cor}[thm]{Corollary}
\newtheorem{lemma}[thm]{Lemma}

\newtheorem{remark}[thm]{Remark}
\newtheorem{defn}[thm]{Definition}
\newtheorem{ex}[thm]{Example}

\begin{document}
\title{The Separation Principle in\\ Stochastic Control, Redux\thanks{This research was supported by grants from
AFOSR, NSF, VR, SSF and the G\"oran Gustafsson Foundation.} }
\author{Tryphon T. Georgiou\thanks{T. T.  Georgiou is with the Department of Electrical \& Computer Engineering,
University of Minnesota, Minneapolis, Minnesota 55455, USA, {\tt tryphon@umn.edu}}
 and Anders Lindquist\thanks{A. Lindquist is with the Department of  Automation, Shanghai Jiao Tong Univerity, Shanghai, China, and the Center for Industrial and Applied Mathematics (CIAM) and the ACCESS Linnaeus Center, 
 Royal Institute of Technology, 100 44 Stockholm, Sweden, {\tt alq@kth.se}}}
\markboth{}
{Georgiou and Lindquist: Separation Principle}
\maketitle

\def\spacingset#1{\def\baselinestretch{#1}\small\normalsize}

\begin{abstract}
Over the last 50 years a steady stream of accounts have been written on the separation principle of stochastic control. Even in the context of the  linear-quadratic  regulator in continuous time with Gaussian white noise, subtle difficulties arise, unexpected by many, that are often overlooked.  In this paper we propose a new framework for establishing the separation principle. This approach takes the viewpoint that stochastic systems are well-defined maps between sample paths rather than stochastic processes {\em per se} and allows us to extend the separation principle to systems driven by martingales with possible jumps. While the approach is more in line with ``real-life'' engineering thinking where signals travel around the feedback loop, it is unconventional from a probabilistic point of view in that control laws 
for which the feedback equations are satisfied almost surely, and not deterministically for every sample path, are excluded.
 \end{abstract}

\section{Introduction}\label{sec:introduction}

\PARstart{O}{ne} of the most fundamental principles of feedback theory is that the problems of optimal control and state estimation can be decoupled in certain cases. This is known as the {\em separation principle}. The concept  was coined early on in  \cite{JosephTou,potter} and is closely connected to the idea of {\em certainty equivalence}; see, e.g., \cite{vanWaterWillems}. In studying the literature on the separation principle of stochastic control, one is struck by the level of sophistication and technical complexity. The source of the difficulties can be traced to the 
circular dependence between control and observations.
The goal of this paper is to present a rigorous approach to the separation principle in continuous time which is rooted in the engineering view of systems as maps between signal spaces.

The most basic setting begins with a linear system
\begin{equation}
\label{eq:system}
\begin{cases}
dx=A(t)x(t)dt+B_1(t)u(t)dt+B_2(t)dw\\
dy=C(t)x(t)dt +D(t)dw
\end{cases}
\end{equation}
with a state process $x$, an output process $y$ and a control $u$,
where $w(t)$ is a vector-valued Wiener process, $x(0)$ is a zero-mean Gaussian random vector independent of $w(t)$, $y(0)=0$, and $A$, $B_1$, $B_2$, $C$, $D$ are matrix-valued functions of compatible dimensions, which we take to be continuous of bounded variation. Moreover, $DD'$ is  nonsingular on the interval $[0,T]$, and if we want the noise processes in the state and output equations to be independent, as often is assumed but not required here, we take $B_2D'\equiv 0$. 
All random variables and processes are defined over a common complete probability space $(\Omega, {\mathcal F}, \mathds{P})$. 

The control problem is to design an output feedback law
\begin{equation}
\label{eq:pi}
\pi \;:\; y \mapsto u
\end{equation}
over the window $[0,T]$ which maps the observation process $y$ to the control input $u$,
in a nonanticipatory manner, so that the value of the functional
\begin{equation}
\label{eq:functional}
J(u) = E\left\{ \int_0^T x(t)'Q(t)x(t)dt+\int_0^Tu(t)'R(t)u(t)dt +x(T)'Sx(T)\right\}
\end{equation}
is minimized,  where $Q$ and $R$ are continuous matrix functions of bounded variation, $Q(t)$ is positive semi-definite and $R(t)$ is positive definite for all $t$. How to choose the admissible class of control laws $\pi$ has been the subject of much discussion in the literature \cite{lindquist}. The conclusion, under varying conditions, has been that $\pi$ can be chosen to be linear in the data and, more specifically, in the form
\begin{equation}
\label{eq:K}
u(t)=K(t)\hat x(t),
\end{equation}
where $\hat x(t)$ is the Kalman estimate of the state vector $x(t)$ obtained from the Kalman filter
\begin{eqnarray}
\label{eq:Kalmanfilter}
d\hat x=A(t)\hat x(t)dt+B_1(t)u(t)dt +L(t)(dy-C(t)\hat x(t)dt),\quad \hat x(0)=0,
\end{eqnarray}
and the gains $K$ and $L$ computed by solving to a pair of dual Riccati equations.

A result of this kind is far from obvious, and the early literature was marred by treatments of the separation principle where the non-Gaussian element introduced by an {\em a priori\/} nonlinear control law $\pi$ was overlooked. The subtlety lies in excluding the possibility that a nonlinear controller extracts more information from the data than it is otherwise possible. This point will be explained in detail in Section \ref{historysec}, where a brief historical account of the problem will be given. Early expositions of the separation principle often fall in one of two categories: either the subtle issues are overlooked and inadmissible shortcuts are taken; or the treatment is mathematically quite sophisticated and technically very demanding. The short survey  in Section \ref{historysec} will thus serve the purpose of introducing the theoretical challenges at hand, as well as setting up notation.

In this paper we take the point of view that feedback laws \eqref{eq:pi} should act on sample paths of the  stochastic process $y$ rather than on the process itself. This is motivated by engineering thinking where systems and feedback loops process signals. Thus, our {\em key\/} assumption on admissible control laws \eqref{eq:pi} is that the resulting feedback loop is {\em deterministically well-posed\/} in the sense that the feedback equations admit a unique solution path-wise which causally depends on the input. For this class of control laws we prove that the  separation principle stated above holds and moreover that it extends to systems driven by general martingale noise. 
More precisely, in this non-Gaussian situation the Wiener process $w$ in \eqref{eq:system} is replaced by an arbitrary martingale process with possible jumps such as  a Poisson process martingale; see, e.g., \cite[p. 87]{Klebaner}. 
Then, we only need to exchange the (linear) Kalman estimate $\hat{x}$ by the strict sense conditional mean
\begin{equation}
\label{eq:xhat}
\hat{x}(t)=E\{ x(t)\mid {\cal Y}_t\},
\end{equation}
where 
\begin{equation}
\label{eq:filtration}
{\cal Y}_t:=\sigma\{ y(\tau), \tau\in [0,t]\}, \quad 0\leq t\leq T,
\end{equation}
is the {\em filtration\/} generated by the output process; i.e., the family of increasing sigma fields representing the data as it is produced. The estimate $\hat x$ needs to be defined with care so that it constitutes a sufficiently regular stochastic process and realized by a map acting on observations \cite[page 17]{BainCrisan}, \cite{CC}.
Unfortunately, the results in the present paper come at a cost since our key assumption of well-posedness excludes control laws for which the feedback system fails to be defined sample-wise. Existence of strong solutions of the feedback equations is not enough to ensure well-posedness in our sense as we will discuss below. In addition, the condition of deterministic well-posedness is often difficult to verify. Yet, besides the fact that we prove the separation principle for general martingale noise, the sample-wise viewpoint provides a simple explanation of why the separation principle may hold in the first place.

Before proceeding we recast the system model \eqref{eq:system} in an integrated form which allows similar conclusions for more general linear systems in a unified setting. To this end, let 
\[z(t)=\begin{pmatrix}
      x(t)    \\
      y(t) 
\end{pmatrix}.\]
System \eqref{eq:system} can now be expressed in the form
\begin{equation}
\label{eq:systemint}
\begin{cases}
z(t)=z_0(t)+\int_0^tG(t,\tau)u(\tau)d\tau\\
y(t)=Hz(t),
\end{cases}
\end{equation}
where  $z_0$ is the process $z$ obtained by setting $u=0$ and $G$ is a Volterra kernel.  This integrated form encompasses a considerably wider class of controlled linear systems which includes delay-differential equations, following \cite{lindquist1,lindquist}, which will be taken up in Section~\ref{delaysec}. 
\begin{figure}[htb]\begin{center}
\psfrag{p1}{$\hspace*{-7pt}\begin{array}{cc}\\[-.12in]\pi \end{array}$}
\psfrag{xx}{$z$}
\psfrag{yy}{$\;y$}
\psfrag{x0}{$z_0$}
\psfrag{plus}{$+$}
\psfrag{g1}{$g$}
\psfrag{h1}{$\hspace*{-9pt}\begin{array}{cc}\\[-.12in]H \end{array}$}
\psfrag{uu}{$u$}
\includegraphics[totalheight=3.5cm]{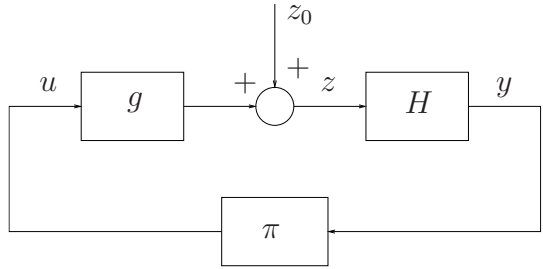}
\caption{A feedback interconnection.}\label{fig1}\end{center} \end{figure}
The corresponding feedback configuration is shown in Figure \ref{fig1} where $g$ represents the Volterra operator
\begin{equation}
\label{G2g}
g\;:\; (t,u) \mapsto \int_0^t G(t,\tau)u(\tau)d\tau,
\end{equation}
and $H$ is a constant matrix. As usual, Figure \ref{fig1} is a graphical representation of  the algebraic relationship
\begin{equation}
\label{z02z}
z=z_0+g\pi H z.
\end{equation}
For the particular model in \eqref{eq:system}, $H=[0,I]$, but in general $H$ could be any matrix or linear system. Setting $z:=x$ and $H=I$ we obtain the special case of complete state information.

In a stochastic setting, the feedback equation \eqref{z02z} is said to have a unique strong solution if there exists a non-anticipating function $F$ such that $z=F(z_0)$ satisfies  \eqref{z02z} with probability one and all other solutions coincide with $z$ with probability one.
It is important to note that in our sample-wise setting we require more, namely that such a unique solution exists and that \eqref{z02z} holds for {\em all\/} $z_0$, not only ``almost all.'' Consequences of this requirement will be further elaborated upon below.

The outline of the paper is as follows. In Section~\ref{historysec} we begin by reviewing the standard quadratic regulator problem and pointing our the subtleties created by possible nonlinearities in the control law. We then review several strategies in the literature to establish a separation principle, chiefly restricting the class of admissible controls.  Section~\ref{signalssec} defines notions of signals and systems  used in our framework, and in Section~\ref{wellposedsec} we establish necessary conditions  for a feedback loop to make sense and deduce a basic fact about propagation of information in the loop through linear components. It Section~\ref{separationsec} we state and prove our main results on the separation principle for linear-quadric regulator problems, allowing also for more general martingale noise. Finally, in Section~\ref{delaysec} we prove a separation theorem for delay systems with Gaussian martingale noise. 

\section{Historical remarks}\label{historysec}

A common approach to establishing the basic separation principle stated at the beginning of Section \ref{sec:introduction} is a completion-of-squares argument similar to the one used in deterministic linear-quadratic-regulator theory; see e.g. \cite{Astrom}. For  ease of reference, we briefly review this contruction.  Given the system \eqref{eq:system} and the solution of the matrix Riccati equation
\begin{subequations}\label{optimalcontrol}
\begin{equation}
\label{eq:Riccati}
\dot{P}=-A'P-PA+PB_1R^{-1}B_1'P-Q, \quad
P(T)=S.
\end{equation}
 It\^o's differential rule (see, e.g., \cite{Klebaner,Oksendal}) yields
\begin{equation*}
\label{eq:Ito}
d(x'Px)=x'\dot{P}xdt+2x'Pdx+ {\rm tr}(B_2'PB_2)dt,
\end{equation*}
where ${\rm tr}(M)$ denotes the trace of the matrix $M$. Then from  \eqref{eq:system} and  \eqref{eq:Riccati} it readily follows that 
\begin{displaymath}
d(x'Px)=[-x'Qx-u'Ru+ (u-Kx)'R(u-Kx)]dt
+{\rm tr}(B_2'PB_2)dt + 2x'PB_2dw,
\end{displaymath}
where
\begin{equation}
\label{eq:gain}
K(t):=-R(t)^{-1}B_1(t)'P(t).
\end{equation}
\end{subequations}
Integrating this from $0$ to $T$ and taking mathematical expectation, we obtain the following expression for the cost functional \eqref{eq:functional}:
\begin{eqnarray}
\label{eq:J(u)}
J(u)&=&E\left\{ x(0)'P(0)x(0) +\int_0^T(u-Kx)'R(u-Kx)dt\right\}
+\int_0^T{\rm tr}(B_2'PB_2)dt.
\end{eqnarray}
To ensure that $\int_0^Tx'PB_2dw$ has zero expectation, we need to check that the integrand is square integrable almost surely.
It is clear that $u$ is square integrable for otherwise $J(u)=\infty$. Then the state process
\begin{equation}
\label{u2x}
x(t)=x_0(t)+\int_0^t \Phi(t,s)B_1(s)u(s)ds
\end{equation}
 is square integrable as well. 
Here  $x_0$ is the (square integrable) state process  corresponding to $u=0$, and $\Phi(t,s)$ is the transition matrix function of the system \eqref{eq:system}.

Now, if we had complete state information with  \eqref{eq:system} replaced by  
\begin{equation}
\label{completeinfsystem}
\begin{cases}
dx=A(t)x(t)dt+B_1(t)u(t)dt+B_2(t)dw\\
y=x
\end{cases}
\end{equation}
we could immediately conclude that the feedback law 
\begin{equation}
\label{completeinformation:pi}
u(t)=K(t)x(t)
\end{equation}
is optimal, because the last term in \eqref{eq:J(u)} does not depend on the control. However, when we have incomplete state information with the control being a function of the observed process $\{y(s); 0\leq s\leq t\}$, things become more complicated. Mathematically we formalize this by having  any control process adapted to the filtration \eqref{eq:filtration}; i.e., having $u(t)$ ${\cal Y}_t$-measurable for each $t\in [0,T]$.  Then, setting
\begin{equation}
\label{eq:xtilde}
\tilde{x}(t):= x(t)-\hat{x}(t)
\end{equation}
with $\hat{x}$ given by \eqref{eq:xhat},
we have $E\{ [u(t)-K(t)\hat{x}(t)]\tilde{x}(t)'\}=0,$ and therefore 
\begin{equation}\label{KRKSigma}
E\int_0^T(u-Kx)'R(u-Kx)dt =E\int_0^T[(u-K\hat{x})'R(u-K\hat{x})+{\rm tr}(K'RK\Sigma)]dt,
\end{equation}
where $\Sigma$ is the covariance matrix
\begin{equation}
\label{eq:Sigma}
\Sigma(t):=E\{\tilde{x}(t)\tilde{x}(t)'\}.
\end{equation}
A common mistake\label{mistake} in the early literature on the separation principle is to assume {\em without further investigation\/} that $\Sigma$ does not depend on the choice of control. Indeed, if this were the case, it would follow directly that \eqref{eq:J(u)} is minimized by choosing the control as \eqref{eq:K}, and the proof of the separation principle would be immediate. (Of course, in the end this will be the case under suitable conditions, but this has to be proven.) This mistake probably originates from the observation that the  control term  in \eqref{u2x}
cancels when forming \eqref{eq:xtilde} so that
\begin{equation}
\label{xtilde}
\tilde{x}(t)= \tilde{x}_0(t):= x_0(t)-\hat{x}_0(t),
\end{equation}
where 
\begin{equation}
\label{eq:xhat0}
\hat{x}_0(t):=E\{ x_0(t)\mid {\cal Y}_t\}.
\end{equation}
However, in this analysis, we have {\em not\/} ruled out that $\hat{x}_0$ depends on the control or, what would follow from this, that the filtration \eqref{eq:filtration} does. A detailed discussion of this conundrum can be found in \cite{lindquist}. In fact, since the control process $u$ is in general a {\em nonlinear\/} function of the data and thus non-Gaussian, then so is the output process $y$.\footnote{However, the model is  conditionally Gaussian given the filtration $\{{\mathcal Y}_t\}$; see Remark~\ref{condGaussianrem}.}  Consequently, the conditional expectation \eqref{eq:xhat0} might not in general coincide with the {\em wide sense\/} conditional expectation obtained by projections of the components of $x_0(t)$ onto the closed linear span of the components of $\{y(\tau), \tau\in [0,t]\}$, and therefore, {\em a priori},  it could happen that $\hat{x}$ is not generated by the Kalman filter \eqref{eq:Kalmanfilter}.

To avoid these problems one might begin by uncoupling the feedback loop as described in Figure~2, and determine an optimal control process in the class of stochastic processes  $u$ that are adapted to the family of sigma fields 
\begin{equation}
\label{Yzerofiltration}
 {\cal Y}_t^0:=\sigma\{ y_0(\tau), \tau\in [0,t]\}, \quad 0\leq t\leq T,
\end{equation}
 i.e., such that,  for each $t\in [0,T]$, $u(t)$ is a function of $y_0(s);\, 0\leq s\leq T$. Such a problem, where one optimizes over the class of all control processes adapted to a fixed filtration,  was called a {\em stochastic open loop (SOL) problem\/} in  \cite{lindquist}. In the literature on the separation principle it is not uncommon to assume from the outset that the control is adapted to $\{ {\mathcal Y}_t^0\}$; see, e.g., \cite[Section 2.3]{Bensoussan}, \cite{vanHandel,willems78}.
\begin{figure}[htb]\begin{center}
\psfrag{p1}{$\hspace*{-8pt}\begin{array}{cc}\\[-.1in]\pi \end{array}$}
\psfrag{xx}{$z$}
\psfrag{yy}{\hspace*{-1pt}$y$}
\psfrag{y0}{\hspace*{-6pt}$y_0$}
\psfrag{x0}{$z_0$}
\psfrag{plus}{$+$}
\psfrag{g1}{$\hspace*{-9pt}\begin{array}{cc}\\[-.14in]g \end{array}$}
\psfrag{h1}{$\hspace*{-9pt}\begin{array}{cc}\\[-.1in]H \end{array}$}
\psfrag{uu}{\hspace*{-3pt}$u$}
\includegraphics[totalheight=2.0cm]{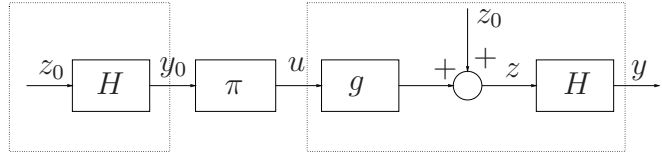}
\caption{A stochastic open loop (SOL) configuration.}\label{fig5}\end{center} \end{figure}

In \cite{lindquist} it was suggested how to embed the class of admissible controls in various SOL classes in a problem-dependent manner, and then construct the corresponding feedback law. More precisely, in the present context, the class of admissible feedback laws was taken to consist of the nonanticipatory functions $u:=\pi(y)$ such that the feedback loop 
\begin{equation}
\label{feedbackeq}
z=z_0+g\pi Hz
\end{equation}
has a unique solution $z_\pi$ and $u=\pi(Hz_\pi)$ is adapted to $\{{\mathcal Y}_t^0\}$. Next, we shall give a few examples of specific classes of feedback laws that belong to this general class. 

\begin{ex}
It is common to restrict the admissible class of control laws to contain only linear ones; see, e.g., \cite{davisbook}.  In a more general direction, let ${\cal L}$ be the class
\begin{equation}
\label{ }
({\cal L})\quad u(t)=\bar{u}_0(t)+\int_0^tF(t,\tau)dy,
\end{equation}
where $\bar{u}$ is a deterministic function and $F$ is an $L_2$ kernel. 
 In this way, the Gaussian property will be preserved, and $\hat{x}$ will be generated by the Kalman filter \eqref{eq:Kalmanfilter}. Then it follows from \eqref{eq:system} and \eqref{eq:Kalmanfilter} that $\tilde{x}$ is generated by
 \begin{displaymath}
d\tilde{x}=(A-LC)\tilde{x}dt +(B_2-LD)dw, \quad \tilde{x}(0)=x(0),
\end{displaymath} 
which is clearly independent of the choice of control. Then so is the error covariance \eqref{eq:Sigma}, as desired. 
Even in the more general setting described by \eqref{eq:systemint}, it was shown in \cite[pp. 95--96]{lindquist1} that  
\begin{equation}
\label{eq:calYconstant}
{\mathcal Y}_t = {\mathcal Y}_t^0, \quad t\in [0,T],
\end{equation}
for any $\pi\in {\cal L}$, where \eqref{Yzerofiltration}
is the filtration generated by the uncontrolled output process $y_0$ obtained by setting $u=0$ in \eqref{eq:systemint}.
\end{ex}

\begin{ex}
In his influential paper \cite{wonham},  Wonham proposed the class of control laws
\begin{equation}
\label{psidef}
u(t)=\psi(t,\hat{x}(t))
\end{equation}
in terms of the state estimate \eqref{eq:xhat}, where $\psi(t,x)$ is Lipschitz continuous in $x$. For pedagogical reasons, we  first highlight a somewhat more restrictive construction due to Kushner \cite{kushner1}.
Let 
\begin{displaymath}
\hat{\xi}_0(t):=E\{ x_0(t)\mid {\mathcal Y}_t^0\}
\end{displaymath}
be the Kalman state estimate of the uncontrolled system
\begin{equation}
\label{eq:uncontroledsystem}
\begin{cases}
dx_0=A(t)x_0(t)dt+B_2(t)dw\\
dy_0=C(t)x_0(t)dt +D(t)dw
\end{cases}
\end{equation}
Here we use the notation $\hat{\xi}_0$ to distinguish it from $\hat{x}_0$, defined by \eqref{eq:xhat0}, which {\em a priori\/} depends on the control. 
Then the Kalman filter takes the form
\begin{displaymath}
d\hat{\xi}_0=A\hat{\xi}_0(t)dt + L(t)dv_0, \; \hat{\xi}_0(0)=0
\end{displaymath}
where the innovation process
\begin{displaymath}
dv_0 = dy_0-C\hat{\xi}_0(t)dt, \; v_0(0)=0
\end{displaymath}
generates the same filtration, $\{ {\mathcal V}_t^0\}$, as $y_0$; i.e., $ {\mathcal V}_t^0={\mathcal Y}_t^0$ for $t\in [0,T]$. This is well-known, but a simple proof is given on page \pageref{innovequiv} in Section~\ref{delaysec} in a more general setting; see  \eqref{equivfiltrations}. Now, along the lines of \eqref{u2x}, define
\begin{displaymath}
\hat{\xi}(t)=\hat{\xi}_0(t)+ \int_0^t \Phi(t,s)B_1(s)u(s)ds,
\end{displaymath}
where the control is chosen as
\begin{equation}
\label{xihat2u}
u(t)=\psi(t,\hat{\xi}(t)).
\end{equation}
Since $\psi$ is Lipschitz,  $\hat{\xi}$ is the unique strong solution of the stochastic differential equation 
\begin{equation}
\label{xiSDE}
d\hat{\xi}=\big(A\hat{\xi}(t) +B_1\psi(t,\hat{\xi}(t))\big)dt + L(t)dv_0,\; \hat{\xi}(0)=0,
\end{equation}
and it is thus adapted to $\{ {\mathcal V}_t^0\}$ and hence to $\{{\mathcal Y}_t^0\}$; see, e.g., \cite[p. 120]{Klebaner}. Hence the selection \eqref{xihat2u} of control law forces $u$ to be adapted to $\{ {\mathcal Y}_t^0\}$, and hence, due to 
\begin{equation}
\label{dy02dy}
dy=dy_0 +\int_0^t C(t)\Phi(t,s)B_1(s)u(s)dsdt,
\end{equation}
obtained from \eqref{u2x}, ${\mathcal Y}_t\subset  {\mathcal Y}_t^0$ for $t\in [0,T]$.  However, since the control-dependent terms cancel,
\begin{displaymath}
dv_0= dy_0- C\hat{\xi}_0(t)dt= dy -C\hat{\xi}(t)dt,
\end{displaymath}
which inserted into \eqref{xiSDE} yields a stochastic differential equation, obeying the appropriate Lipschitz condition, driven by $dy$ and having $\hat{\xi}$ as a strong solution. Therefore,  $\hat{\xi}$ is adapted to $\{ {\mathcal Y}_t\}$, and hence, by \eqref{xihat2u}, so is $u$. Consequently, \eqref{dy02dy} implies that ${\mathcal Y}_t^0\subset {\mathcal Y}_t$  for $t\in [0,T]$ so that actually \eqref{eq:calYconstant} holds. Finally, this implies that $\hat{\xi}=\hat{x}$, and thus $u$ is given by \eqref{psidef}. However, it should be noted that the class of control laws \eqref{xihat2u} is a subclass of \eqref{psidef} as it has been constructed to make $u$ {\em a priori\/} adapted to $\{ {\mathcal Y}_t^0\}$. Therefore, the relevance of these results, presented in \cite{kushner1}, for the proof in \cite[page 348]{kushner} is unclear. In their popular textbook \cite{KwakernaakSivan}, widely used as a reference source for the validity of the separation principle over a general class of admissible (including nonlinear) controls,  Kwakernaak and Sivan prove the separation principle over a class  of  linear laws but claim with reference to \cite{kushner,kushner1}  that it holds ``without qualification" in general  \cite[p. 390]{KwakernaakSivan}. (However, see Remark~\ref{condGaussianrem} below.)

In his pioneering paper \cite{wonham}  Wonham proved the separation theorem for controls in the class   \eqref{psidef} even with a more general cost functional than \eqref{eq:functional}. However, the proof is far from simple and marred by many technical assumptions. A case in point is the assumption that $C(t)$ is square and has a determinant bounded away from zero, which is a serious restriction. A later proof by Fleming and Rishel  \cite{FlemingRishel} is considerably simpler. They also prove the separation theorem with quadratic cost functional  \eqref{eq:functional} for a class of Lipschitz continuous  feedback laws, namely 
\begin{equation}
\label{phidef}
u(t)=\phi(t,y),
\end{equation}
where $\phi:\, [0,T]\times C^n [0,T]\to{\mathbb R}^m$ is a nonanticipatory function of $y$ which is Lipschitz continuous in this argument. 
\end{ex}

\begin{ex}\label{ex3}
It is interesting to note that if there is a delay in the processing of the observed data so that, for each $t$, $u(t)$ is a function of $y(\tau); \, 0\leq\tau\leq t-\varepsilon$, then
\begin{equation}\nonumber
{\cal Y}_t = {\cal Y}_t^0, \quad t\in [0,T].
\end{equation}
To see this,  let $n$ be a positive integer, and suppose that ${\cal Y}_t = {\cal Y}_t^0$ for $t\in[0,n\varepsilon]$. 
Since $u(t)$ is ${\cal Y}_{t-\varepsilon}$-measurable on $[0,(n+1)\varepsilon]$,  it is at the same time ${\cal Y}_{t-\varepsilon}$  as well as ${\cal Y}_{t-\varepsilon}^0$-measurable. 
Then, since
 \begin{displaymath}
y(t)=y_0(t)+\int_0^t HG(t,s)u(s)ds,
\end{displaymath} 
it follows that ${\cal Y}_t = {\cal Y}_t^0$ for $t\in[0,(n+1)\varepsilon]$. Since ${\cal Y}_t = {\cal Y}_t^0$ for $t\in[0,\varepsilon]$, \eqref{eq:calYconstant} follows by induction.
\end{ex}

\begin{remark}
This is the reason why the problem with possibly control-dependent sigma fields does not occur in the usual discrete-time formulation. Indeed, in this setting, the error covariance \eqref{eq:Sigma} will not depend on the control, while, as we have mentioned,  some more analysis is needed to rule out that its continuous-time counterpart does. This invalidates a procedure used in several textbooks (see, e.g., \cite{stengel}) in which the continuous-time $\Sigma$ is constructed as the limit of finite difference quotients of the discrete-time $\Sigma$, which, as we have seen in Example~\ref{ex3}, does not depend on the control, and which simply is the solution of a discrete-time matrix Riccati equation. However, we cannot {\em a priori\/}  conclude that continuous-time $\Sigma$ satisfies  this  Riccati equation. For this we need \eqref{eq:calYconstant}, or alternatively  arguments such as in Remark~\ref{condGaussianrem}.  Otherwise the argument is circular. 
\end{remark}

\begin{remark}\label{varaiya}
Historically, a popular approach was introduced in Duncan and Varaiya, and  Davis and Varaiya  \cite{duncanvaraiya,davisvaraya} based on {\em weak solutions\/}  of the relevant stochastic differential equation. The driving noise is Wiener and the approach utilizes the Girsanov transformation
to recast the problem in a way so that the filtration of the observation process is independent of the input process (see \cite[Section 2.4]{Bensoussan}). Very briefly, by an appropriate  change of probability measure,
\begin{displaymath}
d\tilde{w}=B_1udt + B_2 dw
\end{displaymath}
can be transformed into a new Wiener process, which in the sense of weak solutions \cite{Klebaner} is the same as any other Wiener process.
In this way, the filtration $\{{\mathcal Y}_t\}$ can be fixed to be constant with respect to variations in the control.
In this paper we do not consider weak solutions since our observation process is not arbitrary from an applications point of view.
\end{remark}

\begin{remark}\label{condGaussianrem}
Yet another approach to the separation principle is based on the fact that, although \eqref{eq:system} with a nonlinear control is non-Gaussian, the model is {\em conditionally Gaussian\/} given the filtration $\{{\mathcal Y}_t\}$ \cite[Chapters 16.1] {LiptserShirayev}. This fact can be used to show that $\hat{x}$ is actually generated by a Kalman filter \cite[Chapters 11 and 12] {LiptserShirayev}. This last approach requires quite a sophisticated analysis and is restricted to the case where the driving noise $w$ is a Wiener process.
\end{remark}

A key point for establishing the separation priniciple is to identify admissible control laws for which  \eqref{eq:calYconstant} holds. 
For each such control law $\pi$ we need a solution of the feedback equation \eqref{z02z}, i.e., a pair $(z_0,z)$ of stochastic processes that satisfies
\begin{equation}
 \label{z02z_bis}
 z=z_0+g\pi H z.
\end{equation}
Since $z_0$ is the driving process, it is natural to seek a solution $z$ which causally depends on $z_0$ and is unique. If this is the case then $z$ is a {\em strong solution}; otherwise it is a {\em weak solution}. There are well-known examples of stochastic differential equations that have only weak solutions \cite[page 137]{Klebaner}, \cite{tsirelson,benes}. Moreover, as we have mentioned in Remark \ref{varaiya}, weak solutions circumvent the need to establish the equivalence \eqref{eq:calYconstant} between filtrations. Thus, it has been suggested that
the framework of weak solutions is the appropriate one for control problems \cite[page 149]{RogersWilliams}.
Yet, from an applications point of view, where the control needs to be causally dependent on observed data, this is in our view questionable.
In fact, in the present paper we take an even more stringent view on the causal dependence.
We require that \eqref{z02z_bis} has a unique strong solution which specifies a measurable map $z_0 \to z$ between sample-paths (cf.\ \cite[Remark 5.2, p. 128]{Klebaner}, \cite[p. 122]{RogersWilliams}), thus modeling correspondence between signals -- we further elaborate upon this in Section \ref{wellposedsec}. 

In short, we only allow control laws which are physically realizable in an engineering sense, in that they induce a signal that travels through the feedback loop. This comes at a price since there are stochastic differential equations having strong solutions that do not fall in this category (Remark \ref{failureremark}). Moreover, verifying that a control law is admissible in our sense may be difficult to ascertain in general. On the other hand, an advantage of the approach is that the class of control laws includes discontinuous ones and allows for  statements about linear systems driven by non-Gaussian noise with possible jumps. We now proceed to develop the approach and the key property of {\em deterministic well-posedness}.

\section{Signals and systems}\label{signalssec} 

Signals are thought of as sample paths of a stochastic process with possible discontinuities.
This is quite natural from several points of view. First, it encompasses the response of a typical nonlinear operation that involves  thresholding and switching, and second, it includes sample paths of counting processes and other martingales.
More specifically we consider signals to belong to the {\em Skorohod space} $D$; this is defined as the space of functions which are continuous on the right and have a left limit at all points, i.e., the space of {\em c\`adl\`ag} functions.\footnote{``continu \`a droite, limite  \`a gauche" in French, alternatively RCLL (``right continuous with left limits") in English.} It contains the space $C$ of continuous functions as a proper subspace. The notation $D[0,T]$ or $C[0,T]$ emphasizes the time interval where signals are being considered.

Traditionally, the comparison of two continuous functions in the uniform topology relates to how much their graphs
need to be perturbed so as to be carried onto one another by changing only the ordinates, with the time-abscissa being kept fixed.
However, in order to metrize $D$ in a natural manner one must recognize the effect of uncertainty in measuring time and allow a respective deformation of the time axis as well.
To this end, let ${\mathcal K}$ denote the class of strictly increasing, continuous mappings of $[0,T]$ onto itself and let $I$ denote the identity map. Then, for $x,y\in D[0,T]$,
\[
d(x,y):=\inf_{\kappa\in{\mathcal K}}\max\{\|\kappa - I\|,\|x-y\kappa\|\}
\]
defines a metric
on $D[0,T]$ which induces the so called {\em Skorohod topology}. A further refinement 
so as to ensure bounds on the slopes of the chords of $\kappa$,
renders $D[0,T]$ separable and complete, that is, $D[0,T]$ is a {\em Polish} space; see \cite[Theorem 12.2]{billingsley}.

Systems are thought of as general measurable nonanticipatory maps from $D\to D$ sending sample paths to sample paths so that their  outputs at any given time $t$ is a measurable function of past values of the input and of time. 
More precisely, let
\[
\Pi_\tau\;:\; x\mapsto \Pi_\tau x:=\left\{\begin{array}{l} x(t) \mbox{ for }t<\tau\\ x(\tau) \mbox{ for }t\geq \tau.\end{array}\right.
\]
Then, a measurable map $f:\, D[0,T]\to D[0,T]$ is said to be a {\em system\/} if and only if
\[
\Pi_\tau f\; \Pi_\tau =\Pi_\tau  f \quad \mbox{for all  $\tau\in  [0,T]$.}
\]

An important class of systems is provided by stochastic differential equations with Lipschitz coefficients driven by a Wiener process \cite[Theorem 13.1]{RogersWilliams}. These have path-wise unique strong solutions. Strong solutions induce maps between corresponding path spaces
\cite[page 127]{RogersWilliams}, \cite[pages 126-128]{Klebaner}. Also, under fairly general conditions (see e.g., \cite[Chapter V]{Protter}), stochastic differential equations driven by martingales with sample paths in $D$ have strong solutions who are semi-martingales.

Besides stochastic differential equations in general, and those in \eqref{eq:systemint} in particular, other nonlinear maps may serve as systems. For instance, discontinuous hystereses nonlinearities
as well as non-Lipschitz static maps such as
$u\mapsto y:=\sqrt{|u|}$, are reasonable as systems, from an engineering viewpoint. Indeed, these induce maps from $D\to D$ (or from $C\to D$, as in the case of relay hysteresis), are seen to be systems according to our definition,\footnote{More precisely, to be seen as a system, relay hysteresis needs to be preceded by a low-pass filter since its domain consists of continuous functions.} and can be considered as components of nonlinear feedback laws. 
We note that a nonlinearity such as
$u\mapsto y={\rm sign}(u)$
is not a system in the sense of our definition since the output is not in general in $D$. Such nonlinearities, which often appear in bang-bang control, need to be approximated with a physically realizable hysteretic system. 

\section{Well-posedness and a key lemma}\label{wellposedsec}

It is straightforward to construct examples of deterministically well-posed feedback interconnections with elements as above.
However, the situation is a bit more delicate when considering feedback loops since it is also perfectly possible that, at least mathematically, they give rise to unrealistic behavior. A standard example is that of a feedback loop with causal components that ``implements'' a perfect predictor. Indeed, consider a system $f$ which superimposes  its input with a delayed version of it, i.e.,
\begin{displaymath}
f\;:\; z(t) \mapsto z(t)+z(t-t_{\rm delay}),
\end{displaymath} 
for $t\geq 0$, and assume initial conditions  $z(t)=0$ for $t<0$ . Then the feedback interconnection of Figure \ref{fig2} is unrealistic as it behaves as a perfect predictor. The feedback equation 
\begin{displaymath}
z(t)=z_0(t)+f(z(t))
=z_0(t)+z(t)+z(t-t_{\rm delay})
\end{displaymath}
gives rise to $0=z_0(t)+z(t-t_{\rm delay})$, and hence,
\[
z(t)=-z_0(t+t_{\rm delay}).
\]
Therefore, the output process $z$ is not causally dependent on the input.
The question of well-posedness of feedback systems has been studied from different angles for over forty years. See for instance the monograph by Jan Willems \cite{willems}. 

In our present setting of stochastic control we need a concept of well-posedness which ensures that signals inside a feedback loop are causally dependent on external inputs. This is a natural assumption from a systems point of view.

\begin{defn}
A feedback system is  {\em deterministically well-posed} if the closed-loop maps are themselves systems; i.e., the feedback equation $z=z_0+f(z)$
has a unique solution $z$ for inputs $z_0$ and the operator $(1-f)^{-1}$ is itself a system.
\end{defn}
\begin{figure}[htb]\begin{center}
\psfrag{p1}{{\Large\hspace*{-6pt}$\;_{f}$}}
\psfrag{xx}{$z$}
\psfrag{yy}{$\;y$}
\psfrag{x0}{$z_0$}
\psfrag{plus}{$+$}
\psfrag{g1}{$g$}
\psfrag{h1}{$\hspace*{-5pt}H$}
\psfrag{uu}{$u$}
\includegraphics[totalheight=3cm]{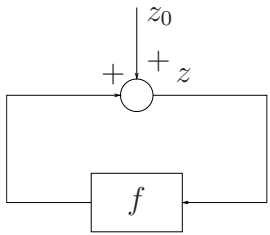}
\caption{Basic feedback system.}\label{fig2}\end{center} \end{figure}

Thus, now thinking about  $z_0$ and  $z$ in the feedback system  in Figure \ref{fig2} as stochastic processes, deterministic well-posedness implies that ${\mathcal Z}_t\subset {\mathcal Z}^0_t$ for $t\in [0,T]$, where ${\mathcal Z}_t$ and ${\mathcal Z}^0_t$ are the sigma-fields generated by $z$ and $z_0$, respectively. This is a consequence of  the fact that $(1-f)^{-1}$ is a system. Likewise, since $(1-f)$ is also a system, ${\mathcal Z}_t^0\subset {\mathcal Z}_t$ so that in fact 
\begin{equation}
\label{well-posednessass}
{\mathcal Z}_t^0= {\mathcal Z}_t, \quad t\in [0,T].
\end{equation}

Next we consider the situation in Figure \ref{fig1} and the relation between ${\mathcal Y}_t$ and the filtration ${\mathcal Y}^0_t$ of the process $y_0=Hz_0$.  The latter represents the ``uncontrolled'' output process where the control law $\pi$ is taken to be identically zero. A key technical lemma for what follows states that the filtrations ${\mathcal Y}_t$ and ${\mathcal Y}^0_t$ are also identical if the feedback system is deterministically well-posed. This is not obvious at first sight, solely on the basis of the linear relationships $y=Hz$ and $y_0=Hz_0$, as the following simple example demonstrates: the two vector processes 
${w\choose 0}$ and ${0\choose w}$
generate the same filtrations while
$(1\; 0){w\choose 0}$ and $(1\; 0){0\choose w}$
do not.

\begin{lemma}\label{keylemma}
If the feedback interconnection in Figure~\ref{fig1} is deterministically well-posed, $g\pi$ is a system,
and $H$ is a linear system having a right inverse $H^{\rm -R}$ that is also a  system, then
$(1-Hg\pi)^{-1}$ is a system and ${\mathcal Y}_t={\mathcal Y}^0_t, \quad t\in [0,T]$.
\end{lemma}

\begin{remark}
Note that, for the prototype problem involving \eqref{eq:system}, the conditions  on $H$ in Lemma~\ref{keylemma} are trivial  as 
 $H=\begin{bmatrix}0&I\end{bmatrix}$ and hence  $H^{\rm -R}:=H'$ is a right inverse. The requirement in the lemma
that $g\pi$ is a system allows for a more general situation where $\pi$ is not itself a system (e.g., generating outputs not in D), but where the cascade connection is still admissible.
\end{remark}

\begin{proof}
By well-posedness $(1-g\pi H)^{-1}$ is a system. To show that $(1-Hg\pi)^{-1}$ exists and is a system, first note that
\begin{equation}
\label{eq:distributivity}
(1-Hg\pi)H = H-Hg\pi H= H (1-g\pi H).
\end{equation}
The first step is using left distributivity and the second is using the fact that $H$ is linear. But then
\begin{equation}
\label{rightinverse}
(1-Hg\pi)\underbrace{H(1-g\pi H)^{-1}H^{\rm -R}}_{h}=I,
\end{equation}
where  $HH^{\rm -R}=I$.
Thus, $h$ is a ``right inverse'' of $p:=(1-Hg\pi)$ in that the composition $p\circ h$ of the two maps is the identity.
We claim that $h$ is in fact the inverse of $p$ (which is necessarily unique) in that $y=h(y_0)$ and
\begin{equation}\label{eq:yy0}
(1-Hg\pi)y=y_0
\end{equation}
establish a bijective correspondence between $y$ and $y_0$, i.e., that both $p\circ h$ as well as $h\circ p$ are identity maps. We need to show the latter.
The only potential problem would be if two distinct values $y$ and $\hat y$ satisfy \eqref{eq:yy0} for the same value for $y_0$. We now show that this is not possible.

Since $H$ is right invertible, $y_0$ can be written in the form
$y_0=Hz_0$ for $z_0=H^{- \rm R}y_0$. Let $z=(1-g\pi H)^{-1}z_0$ and $y=Hz$. Then $y=h(y_0)$, so by \eqref{rightinverse} $y$ is  a particular solution of equation \eqref{eq:yy0}. Now let  $\hat y$ be another solution, i.e., suppose that
\begin{equation}\label{eq:yy1}
(1-Hg\pi)\hat y=y_0
\end{equation}
and that $\hat y\neq y$.
We begin by writing $\hat y$
in the form $\hat y=H\hat z$, which can always be done since $H$ is right invertible. Next we set $\hat z_0:=(1-g\pi H)\hat z$. Then, by well-posedness, $\hat{z}$ is the unique solution of 
\begin{equation}
\label{zhatfeedback}
\hat z=\hat z_0+g\pi H(\hat z).
\end{equation}
Moreover, by \eqref{eq:distributivity} and \eqref{eq:yy1},  $H\hat z_0=y_0$, and consequently $\hat z_0=z_0+v$ with $Hv=0$.
We now claim that $\hat z=z+v$ which would then contradict the assumption that $\hat y\neq y$. To show this,  note that, since $z=z_0+g\pi H z$ and $H$ is linear, 
\[
z+v=z_0+v+g\pi H(z+v).
\]
But the solution to \eqref{zhatfeedback} is unique by well-posedness. Hence,  $\hat z=z+v$ which proves our claim.

Therefore, finally, $(1-Hg\pi)$ is invertible and
\begin{displaymath}
(1-Hg\pi)^{-1}=h
=H(1-g\pi H)^{-1}H^{\rm -R}
\end{displaymath}
is itself is a system, being a composition of systems. 
Thus, the configuration in Fig.\ \ref{fig4} is deterministically well-posed.
Using \eqref{eq:distributivity} once again,
\begin{equation}
\label{eq:commutation}
H(1-g\pi H)^{-1}=(1-Hg\pi)^{-1}H.
\end{equation}
It now follows that
\begin{equation}
\label{eq:first}
y = H(1-g\pi H)^{-1}z_0 
   =  (1-Hg\pi)^{-1}Hz_0
  = (1-Hg\pi)^{-1}y_0,
\end{equation}
while also \eqref{eq:yy0} holds.
Equation \eqref{eq:first} shows that ${\mathcal Y}_t\subset {\mathcal Y}^0_t$, whereas \eqref{eq:yy0} shows that
${\mathcal Y}^0_t\subset {\mathcal Y}_t$.
\end{proof}

The essence of the lemma\footnote{It is interesting to note, as was pointed out by a referee, that the proof of the lemma relies critically on the action of the operator $(1-g\pi H)^{-1}$ on a null set, as the probability ${\mathbb P}(z_0=H^{-R}y_0)=0$ for any nontrivial model. This fact may be disturbing from a probabilistic point of view but does not invalidate the lemma.} is to underscore the equivalence between the configuration in Figure \ref{fig1} and that in Figure \ref{fig4}. It is this equivalence  which accounts for the identity ${\mathcal Y}_t={\mathcal Y}^0_t$ between the respective $\sigma$-algebras. An analogous notion of well-posedness was considered by Willems in \cite{willems78} where however, in contrast, the  well-posedness of the feedback configuration in Figure \ref{fig4}, and consequently the validity of
 ${\mathcal Y}_t={\mathcal Y}^0_t$, is assumed at the outset. 
\begin{figure}[htb]\begin{center}
\psfrag{p1}{{\Large\hspace*{-6pt}$\;_{\pi}$}}
\psfrag{xx}{$\;$}
\psfrag{yy}{$\;y$}
\psfrag{x0}{$\hspace*{-5pt}z_0$}
\psfrag{plus}{$\hspace*{-5pt}+$}
\psfrag{g1}{$g$}
\psfrag{h1}{$\hspace*{-9pt}\begin{array}{cc}\\[-.12in]H \end{array}$}
\psfrag{uu}{$u$}
\includegraphics[totalheight=4.5cm]{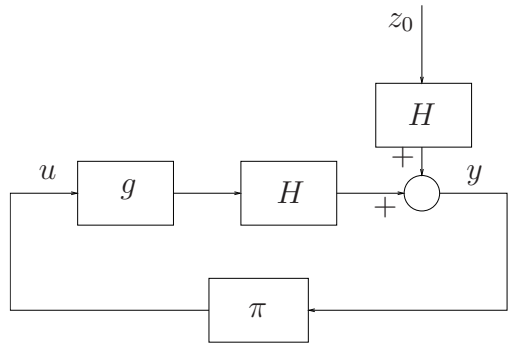}
\caption{An equivalent feedback configuration.}\label{fig4}\end{center} \end{figure}

In the present paper we consider only feedback laws that render the feedback system deterministically well-posed.
Therefore we highlight the conditions in a formal definition.

\begin{defn}
A feedback law $\pi$ is {\em deterministically well-posed\/} for the system \eqref{eq:systemint} if  $g \pi$ is a system and the feedback loop of  Figure \ref{fig1} is deterministically well-posed.
\end{defn}

If the feedback law $\pi$ is deterministically well-posed, then, by Lemma~\ref{keylemma}, the feedback loop in Figure~\ref{fig4} is also deterministically well-posed. Thus, in essence, given the assumption that $z=z_0+g\pi H z$ admits a pathwise unique strong solution, so does $y=y_0+Hg\pi y$.

\begin{remark}\label{rem:completeinf}
For pedagogical reasons, we consider the case of complete state information, corresponding to 
\eqref{completeinfsystem}. 
This corresponds to taking $H=I$ and $z=x$, and the basic feedback loop is as depicted in Figure~\ref{fig2mod}. Then the basic condition for well-posedness \eqref{well-posednessass} states that the filtration $\{ {\mathcal X}_t\}$, where  ${\mathcal X}_t:=\sigma\{ x(s);\; s\in [0,T]\}$, is
\begin{figure}[htb]\begin{center}
\psfrag{p1}{{\Large\hspace*{-3pt}$\;_{f}$}}
\psfrag{xx}{$x$}
\psfrag{yy}{$\;y$}
\psfrag{x0}{$x_0$}
\psfrag{plus}{$+$}
\psfrag{g1}{$g$}
\psfrag{h1}{$\hspace*{-5pt}H$}
\psfrag{uu}{$u$}
\includegraphics[totalheight=3.5cm]{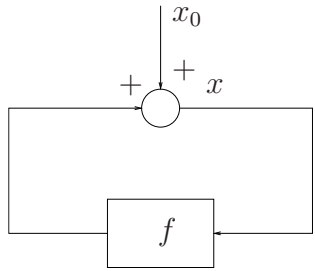}
\caption{Feedback loop for complete state information.}\label{fig2mod}\end{center} \end{figure}
 constant under variations of the control. Consequently, we do not need Lemma~\ref{keylemma} to resolve an issue of circular control dependence. This is completely consistent with the analysis leading up to  \eqref{completeinformation:pi} in Section~\ref{historysec}. 

\end{remark}

\begin{remark}\label{failureremark}
We now present an example of a feedback system which fails to be deterministically well-posed\footnote{This example was kindly suggested by a referee.}. Consider the system
\begin{equation}\nonumber
\begin{cases}
dx=udt + dw\\
y= x
\end{cases}
\end{equation}
where $w$ is a Wiener process. Then, the control law $u=\pi(y)$ with $\pi(y)=\max\{|x|^{2/3},1\}$ is not deterministically well-posed although the stochastic differential equation
\[
dx=\pi(x)dt+dw
\]
has a unique strong solution \cite[Chapter 5, Proposition 5.17]{KaratzasShreve} in the sense that any other solution has same sample paths with probability one (indistinguishable). The failure to be deterministically well-posed can be traced to the fact that this control law allows for multiple consistent responses for $w\equiv 0$, a physically questionable situation. Indeed, the ordinary differential equation $\dot x =\pi(x)$ is not Lipschitz and has infinitely many solutions.
\end{remark}

\section{The separation principle}\label{separationsec}

Our first result is  a very general separation theorem  for the classical stochastic control problem  stated at the beginning of Section \ref{sec:introduction}.

\begin{thm}\label{basicsepthm}
Given the system \eqref{eq:system}, consider the problem of minimizing the functional \eqref{eq:functional} over the class of all feedback laws $\pi$ that are deterministically well-posed for \eqref{eq:system}. Then the unique optimal control law is given by \eqref{eq:K}, where $K$ is defined by \eqref{optimalcontrol}, and $\hat{x}$ is given by the Kalman filter \eqref{eq:Kalmanfilter}. 
\end{thm}

\begin{proof}
By  Lemma~\ref{keylemma}, \eqref{eq:Sigma} does not depend on the control. Therefore, given the analysis at the beginning of Section~\ref{historysec}, \eqref{eq:K} is the unique optimal control provided it defines a deterministically well-posed control law. It remains to show this.

Inserting \eqref{eq:K} into \eqref{eq:Kalmanfilter} yields 
\begin{displaymath}
\hat{x}(t)=\int_0^t\Psi(t,s)L(s)dy(s),
\end{displaymath}
where the transition matrix $\Psi(t,s)$ of $[A(t)+B_1(t)K(t)-L(t)C(t)]$ has partial derivatives in both arguments.  Together with \eqref{eq:K} this yields
\begin{equation}
\label{pioptdy}
u(t)=(\pi_{\rm opt}y)(t):=\int_0^t M(t,s)dy(s), 
\end{equation}
where $M(t,s):=K(t)\Psi(t,s)L(s)$. Clearly  $s\mapsto M(t,s)$ has bounded variation for each $t\in [0,T]$, and therefore integration by parts yields
\begin{equation}
\label{piopty}
(\pi_{\rm opt}y)(t)=  M(t,t)y(t) -\int_0^td_s M(t,s)y(s)ds,
\end{equation}
which  is defined samplewise.  Now inserting $u=\pi_{\rm opt}Hz$ into \eqref{G2g} and \eqref{z02z} we obtain
\begin{equation}
\label{optequ}
z=z_0+g\pi_{\rm opt}Hz,
\end{equation}
where $g\pi_{\rm opt}Hz$ takes the form
\begin{displaymath}
(g\pi_{\rm opt}Hz)(t) =\int_0^t N(t,s)dz(s) \quad \mbox{with\;\;} N(t,s)=\int_s^tG(t,\tau)M(\tau,s)Hd\tau ,
\end{displaymath}
where $G$ is the kernel of the Volterra operator \eqref{G2g}. A simple calulation yields
\begin{displaymath}
\frac{\partial G}{\partial s}(t,s)=\begin{bmatrix}A(t)\\C(t)\end{bmatrix}\Phi(t,s)B_1(s),
\end{displaymath}
where $\Phi(t,s)$ is the transition matrix of $A$, and therefore 
$Q(t,s):=\frac{\partial N}{\partial s}(t,s)$
is a continuous Volterra kernel, and so is the unique solution $R$ of the resolvent equation 
\begin{equation}
\label{resolventeq}
R(t,s)=\int_s^tR(t,\tau)Q(\tau,s)d\tau + Q(t,s)
\end{equation}
\cite{Smithies,Zabreyko}. From \eqref{optequ} we have
\begin{displaymath}
dz =dz_0 +\int_0^tQ(t,s)dz(s)dt
\end{displaymath}
from which it follows that
\begin{displaymath}
\int_0^tQ(t,s)dz(s) = \int_0^tR(t,s)dz_0(s).
\end{displaymath}
Consequently,  $(1-g\pi_{\rm opt}H)$ has a unique preimage given by
\begin{displaymath}
[(1-g\pi_{\rm opt}H)^{-1}z](t)=z_0(t)+\int_0^t\int_\tau^t R(t,s)dsdz_0(\tau),
\end{displaymath}
which is clearly a system, as claimed. Hence the feedback loop is deterministically well-posed.
\end{proof}

Consequently, for a system driven by a Wiener process with Gaussian initial condition, the linear control law defined by \eqref{eq:K} and \eqref{eq:Kalmanfilter} is optimal in the class of all linear and nonlinear control laws for which the feedback system is deterministically well-posed.

If we forsake the requirement that $\hat{x}$ is given by the Kalman filter \eqref{eq:Kalmanfilter}, we can now allow $x_0$ to be non-Gaussian and $w$ to be an arbitrary martingale, even allowing jumps.

\begin{thm}\label{thm:martingale}
Given the system \eqref{eq:system}, where $w$ is a martingale and  $x(0)$ is an arbitrary zero mean random vector independent of $w$, consider the problem of minimizing the functional \eqref{eq:functional} over the class of all feedback laws $\pi$ that are deterministically well-posed for \eqref{eq:system}. Then, provided it is deterministically well-posed,  the unique optimal control law is given by \eqref{eq:K}, where $K$ is defined by \eqref{optimalcontrol} and $\hat{x}$ is the conditional mean \eqref{eq:xhat}.
\end{thm}

\begin{proof}
Given Lemma~\ref{keylemma}, we can use the same completion-of-squares argument as in Section~\ref{historysec} except that we now need to use Ito's differential rule for martingales (see, e.g.,  \cite{Klebaner,Protter}), which, in integrated form, becomes 
\begin{equation}\label{generalIto}
\begin{split}
x(T)'P(T)x(T)-x(0)'P(0)x(0)=f_\Delta +\phantom{xxxxxxxxxxxxxxxx}\\+ \int_0^T\{ x(t)'\dot{P}(t)x(t)dt+2x(t_-)'P(t)dx+ {\rm tr}\left(  P(t)d[x,x'] \right)\} ,
\end{split}
\end{equation} 
where $[x,x']$ is the quadratic variation of $x$  and $f_\Delta$ is an extra term which is in general nontrivial when $w$ has a jump component. Now let
\begin{displaymath}
q(t):=\int_0^t \Phi(t,s)\big(A(s)x(s)+B_1(s)u(s)\big) ds,
\end{displaymath}
where  $\Phi$ is the transition function of \eqref{eq:system} which is differentiable in both arguments. Then,  $x=q+v$,
where $dv=B_2dw$ and $q$ is a continuous process with bounded variation. Therefore
\begin{displaymath}
[x,x']=[q,q']+2[q,v']+[v,v']=[v,v'].
\end{displaymath}
In fact, $[q,q']=[q,v']=0$ \cite[Corollary 8.5]{Klebaner}. Since $v$ does not depend on the control $u$, neither does the last term in the integral in \eqref{generalIto}. If $w$ has a jump component, we have  a nontrivial extra term in \eqref{generalIto}, namely
\begin{displaymath}
f_\Delta=\sum_{s\leq T} \big[ x(s)'P(s)x(s)-x(s_-)'P(s)x(s_-) -2x(s_-)'P(s)\Delta_s  - \Delta_s' P(s)\Delta_s\big]
\end{displaymath}
where the sum is over all jump times $s$ on the interval $[0,T]$ and $\Delta_s:=x(s)-x(s_-)$ is the jump, and we need to ensure that this term does not depend on the control  either. However, since $x(s)=x(s_-)+\Delta_s$, we have $f_\Delta=0$. 

Then the rest of the proof that  \eqref{eq:K} with $\hat{x}$ given by \eqref{eq:xhat} is the unique minimizer of \eqref{eq:functional} over all deterministically well-posed control laws follows from an argument as in Section~\ref{historysec}. More precisely, using \eqref{optimalcontrol} and completing the squares we obtain
\begin{equation}
\label{eq:functional1}
\begin{split}
\int_0^T x(t)'Q(t)x(t)dt+\int_0^Tu(t)'R(t)u(t)dt +x(T)'Sx(T)\\
= x(0)'P(0)x(0)+\int_0^T(u-Kx)'R(u-Kx)dt\\
+\int_0^T {\rm tr}\left(  P(t)d[v,v']\right) +\int_0^Tx(t_-)'P(t)B_2(t)dw.
\end{split}
\end{equation}
Next we show that 
$E\left\{\int_0^Tx(t_-)'P(t)B_2(t)dw\right\}$ does not depend on the control in this more general case as well.
Since only the second term in \eqref{u2x} depends on the control, the problem reduces to showing that
\[
E\left\{\int_0^T\left[ \int_0^t\Phi(t,s)B_1(s)u(s)ds\right]^\prime P(t)B_2(t)dw(t)\right\}=0.
\]
After a change in the order of integration this is equivalent to
\begin{equation}
\label{crossterm}
E\left\{\int_0^T u(t)'\int_t^TF(t,s)dw(s)dt\right\}=0,
\end{equation}
where $F(t,s)=B_1(t)'\Phi(s,t)'P(s)B_2(s)$. However, since $u(t)$ is ${\mathcal W}_t$-measurable, where ${\mathcal W}_t$ is the sigma-field generated by $\{ w(s); 0\leq s\leq t\}$, \eqref{crossterm} can be written
\begin{displaymath}
E\left\{\int_0^T u(t)'E\left\{\int_t^TF(t,s)dw(s)\mid{\mathcal W}_t\right\}dt\right\},
\end{displaymath}
which is zero since $w$ is a martingale. In view of \eqref{KRKSigma} where  \eqref{eq:Sigma} does not depend on the control (Lemma \ref{keylemma}) the statement of the theorem follows. \end{proof}

We note that in general the optimal control law does not belong to ${\cal L}$ and that $\hat{x}$ is not given by the Kalman filter \eqref{eq:Kalmanfilter} but by the conditional mean \eqref{eq:xhat}, which then has to be chosen with some care since it is only defined almost surely as projection for each individual time $t$. To this end it is standard to select the optional projection of $x(t)$ on ${\mathcal Y}_t$ which is a stochastic process with a c\`adl\`ag version \cite[page 17]{BainCrisan}.
Often $\hat x$ is given by a nonlinear filter as in the following example. However, even in those cases, it is difficult to ascertain well-posedness. At present, we are unable to establish that the control law in the example is deterministically well-posed and hence optimal in our admissible class of controls. We conjecture that Theorem \ref{thm:martingale} can be strengthened by removing the {\em a priori} assumption of well-posedness for the case where the optimal filter can be expressed as a stochastic differential equation with locally Lipschitz coefficients by suitable use of stopping times. Such a strengthening would suffice to prove optimality for the following example where we are currently unable to prove well-posedness.

\begin{figure}[htb]\begin{center}
\psfrag{v}{$v$}
\psfrag{u}{$u$}
\psfrag{ydot}{$\dot y$}
\psfrag{x}{$x$}
\psfrag{int}{$\Large \int$}
\psfrag{wdot}{$\sigma\dot w$}
\psfrag{h1}{$\hspace*{-9pt}\begin{array}{cc}\\[-.11in]H \end{array}$}
\psfrag{uu}{$u$}
\includegraphics[totalheight=1.8cm]{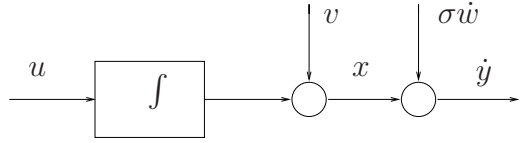}
\caption{Model for step change in white noise.}\label{fig:industrial}\end{center} \end{figure}

\begin{ex} 
Consider the system in Figure \ref{fig:industrial}. Here, $x$ represents a parameter which undergoes a sudden random step change due to a random external forcing $v$. The step can be in either direction. Thus, as a stochastic process $v(t)$ is defined
\begin{equation}
\label{process_v}
v(t)=\begin{cases}
  \theta    &  t\geq\tau \\
  0    &  t<\tau
\end{cases}
\end{equation}
where $\theta=\pm 1$ with equal probability and $\tau$ is a random variable uniformly distributed on $[0,\,T]$. Clearly $v$ is a martingale. Our goal is to maintain a value for the state $x$ close to zero on the interval $[0,T]$ via integral control action through $u$, indirectly, by demanding that
\begin{displaymath}
E\left\{\int_0^T (x^2+R u^2) dt\right\}
\end{displaymath}
be minimal with $R>0$. Here, $u$ denotes the control. The process $x$ is observed in additive white noise $\dot w$.
The system is now written in the standard form \eqref{eq:system} as follows:
\begin{equation}
\label{eq:system_example}
\begin{cases}
dx=u(t)dt+dv,\; x(0)=0,\\
dy=x(t)dt +\sigma dw
\end{cases}
\end{equation}
where $w$ is a Wiener process. We first solve the Riccati equation
$\dot k=-k^2+R^{-1}$ with boundary condition $k(T)=0$ to obtain $k(t)=-R^{-1/2}\tanh\left(R^{-1/2}(T-t)\right)$.
The control law in Theorem \ref{thm:martingale} is
\begin{subequations}\label{control_law}
\begin{equation}
u(t)=k(t)\hat x(t),
\end{equation}
where the conditional expectation
is determined separately using a (nonlinear) Wonham-Shiryaev filter
\begin{eqnarray}
d\hat x &=& k(t) \hat x(t) dt + \frac{1}{\sigma^2}(1-\rho(t)^2- 2(T-t)\phi(t))(dy-\hat x(t) dt)\\
d\rho &=& \frac{1}{\sigma^2}(1-\rho(t)^2- 2(T-t)\phi(t))(dy-\hat x(t) dt)\\
d\phi&=& -\frac{1}{\sigma^2}\phi(t)\rho(t) (dy-\hat x(t) dt)
\end{eqnarray}
\end{subequations}
with $\rho(0)=0$ and $\phi(0)=1$.  Following \cite[page 222]{vanHandel} we explain the steps for deriving the filter equations in Appendix \ref{Appendix_Example}.

In order to conclude that the control law \eqref{control_law}
is {\em actually} optimal we need to establish that the feedback loop is deterministically well-posed. This requires that
\eqref{z02z} has a unique solution for each
$z_0=\left(\begin{matrix} v& w\end{matrix}\right)^\prime$. Noting that the inovation $dy-\hat xdt$ can be expressed as
\[dy-\hat x(t)dt=(v(t)-\rho(t))dt+dw,\]
this requires that the stochastic differential equations
\eqref{eq:system_example}-\eqref{control_law}
can be uniquely solved path-wise as a map from $z_0=\left(\begin{matrix} v& w\end{matrix}\right)^\prime$ to $z=\left(\begin{matrix} x& y\end{matrix}\right)^\prime$. There are conditions in the literature for when path-wise uniqueness holds (see \cite[page 126, Theorem 10.4]{RogersWilliams},  \cite[page 128]{Klebaner}, and the references therein). However, we are not able at present to verify that these hold in our case.
\end{ex}

In view of  Remark~\ref{rem:completeinf} we immediately have the following corollary to Theorem~\ref{thm:martingale} for the case of complete state information. A similar statement was given in \cite{lindquist} in a different context.

\begin{cor}\label{cor:completeinfo}
Given the system \eqref{completeinfsystem}, where $w$ is a martingale and $x(0)$ is an arbitrary random vector independent of $w$, consider the problem of minimizing the functional \eqref{eq:functional} over the class of all feedback laws $\pi$ that are deterministically well-posed for \eqref{completeinfsystem}. Then the unique optimal control law is given by \eqref{completeinformation:pi}, where $K$ is defined by \eqref{optimalcontrol}. 
\end{cor}

\begin{proof}
It just remains to prove that the control law  \eqref{completeinformation:pi} is deterministically well-posed. To this end, we first note that (with $z=x$) the feedback equation \eqref{z02z} becomes 
\begin{displaymath}
x(t) =x_0(t) + \int_0^t Q(t,s)x(s)ds,
\end{displaymath}
where $Q(t,s)=\Phi(t,s)B_1(s)K(s)$ with $\Phi$ (as before) being the  transition matrix function of $A$. Then a straight-forward calculation shows that 
\begin{displaymath}
x(t) =x_0(t) + \int_0^t R(t,s)x_0(s)ds,
\end{displaymath}
where $R$ is the unique solution of the resolvent equation \eqref{resolventeq}. This establishes well-posedness. 
\end{proof}

\begin{ex}
Let the driving noise $w$ in  \eqref{completeinfsystem} be given by either a Poisson martingale \cite[page 87]{Klebaner}, or a {\em geometric Brownian motion} \cite[page 124]{Klebaner}
\begin{displaymath}
dw=\mu w(t)dt + \sigma w(t)dv,
\end{displaymath}
where $v$ is a Wiener process, or a combination. 
Then the control law $u(t)=K(t)x(t)$ is optimal for the problem to minimize \eqref{eq:functional}. 
\end{ex}

\section{The separation principle for delay-differential systems}\label{delaysec}

The formulation \eqref{eq:systemint} covers more general stochastic systems than the ones considered above. An example is a delay-differential system of the type
\begin{displaymath}
\begin{cases}
dx=A_1(t)x(t)dt +A_2(t)x(t-h)dt\ +\int_{t-h}^t A_0(t,s)x(s)dsdt + B_1(t)u(t)dt + B_2(t)dw
\\
dy=C_1(t)x(t)dt+C_2(t)x(t-h)dt +D(t)dw
\end{cases}
\end{displaymath}
Apparently, stochastic control for various versions of such systems were first studied in \cite{L68,L69,lindquist1,lindquist}, and \cite{Brooks}, although \cite{Brooks} relies on the strong assumption that the observation $y$ is ``functionally independent" of the control $u$, thus avoiding the key question studied in the present paper. 

Here, as in \cite{lindquist1}, we shall  consider the wider class of stochastic systems
\begin{equation}
\label{delaysystem}
\begin{cases}
dx=\left(\int_{t-h}^t d_sA(t,s)x(s)\right) dt + B_1(t)u(t)dt+B_2(t)dw\\
dy=\left(\int_{t-h}^t d_sC(t,s)x(s)\right) dt +D(t)dw
\end{cases}
\end{equation} 
where $A$ and $C$ are of bounded variation in the first argument and continuous on the right in the second, $x(t)=\xi(t)$ is deterministic (for simplicity) for $-h\leq t\leq 0$, and $y(0)=0$. 
More precisely, $A(t,s)=0$ for $s\geq t$, $A(t,s)=A(t,t-h)$ for $t\leq t-h$, and the total variation of $s\mapsto A(t,s)$ is bounded by an integrable function in the variable $t$§, and the same holds for $C$.
Moreover, to avoid technicalities we assume that $w$ is now a (square-integrable) Gaussian (vector) martingale. Now,  the first of equations \eqref{delaysystem} can be written in the form
\begin{equation}
\label{integratedxdelay}
\begin{split}
x(t)=\Phi(t,0)\xi(0)+\int_{-h}^0 d_\tau\left\{\int_0^t\Phi(t,s)A(s,\tau)ds\right\}\xi(\tau)\\
+\int_0^t\Phi(t,s)B_1(s)u(s)ds + \int_0^t\Phi(t,s)B_2(s)dw
\end{split}
\end{equation}
\cite[p.\ 85]{lindquist1}, where $\Phi$ is the Green's function corresponding to the determinisitic system \cite{Banks} (also  see, e.g.,  \cite[p.\ 101]{lindquist1}). In the same way, we can express the second equation in integrated form. Consequently, \eqref{delaysystem} can be written in the form  \eqref{eq:systemint}, where $K$ and $H$ are computed as in \cite[pp.\ 101--103]{lindquist1}.
The problem is to find a feedback law \eqref{eq:pi} that minimizes 
\begin{equation}
\label{Jalpha}
J(u):=E\{ V_0(x,u)\}
\end{equation}
subject to the constraint \eqref{delaysystem}, where
\begin{equation}
\label{V_s}
V_s(x,u):=\left\{ \int_s^T x(t)'Q(t)x(t)d\alpha(t)+\int_s^Tu(t)'R(t)u(t)dt\right\}
\end{equation}
and $d\alpha$ is a positive Stieltjes measure.  

Lemma~\ref{keylemma} enables us to strengthen the results in \cite{lindquist1}. To this end, to avoid technicalities, we shall appeal to a representation result from  \cite{lindquist} rather than using a completion-of-squares argument, although the latter strategy would lead to a stronger result where $w$ could be an arbitrary martingale. A completion-of-squares argument for a considerably simpler problem was given in \cite{Blankenship}, but, as pointed out in \cite{L80comment}, this paper suffers from a  similar mistake as the one pointed out earlier on page~\pageref{mistake} in the present paper.
In this context, we also mention the recent paper \cite{BRZ}, which considers optimal control of a stochastic system with delay in the control. This paper assumes at the outset that the separation principle for delay systems  is valid with a reference to \cite{KwakernaakSivan}. Instead of basing the argument on \cite{KwakernaakSivan}, which is not quite appropriate here, their claim could  be justified by noting that the delay in the control also implies a delay in information as in Example~\ref{ex3} above.

Now, it can be shown that the corresponding deterministic control problem obtained by setting $w=0$ has an optimal linear feedback control law
\begin{equation}
\label{delaycontrol}
u(t)=\int_{t-h}^t d_\tau K(t,\tau)x(\tau),
\end{equation}
where we refer the reader to \cite{lindquist1} for the computation of $K$. The following theorem is a considerable strengthening of the corresponding result in \cite{lindquist1}.

\begin{thm}\label{delaythm}
Given the system \eqref{delaysystem}, where $w$ is a Gaussian martingale,  consider the problem of minimizing the functional \eqref{Jalpha} over the class of all feedback laws $\pi$ that are deterministically well-posed for \eqref{eq:system}. Then the unique optimal control law is given by 
\begin{equation}
\label{delaystochasticcontrol}
u(t)=\int_{t-h}^t d_sK(t,s)\hat{x}(s|t),
\end{equation}
where $K$ is the deterministic control gain \eqref{delaycontrol} and 
\begin{equation}
\label{xhat(s|t)}
\hat{x}(s| t) := E\{ x(s)\mid {\cal Y}_t\}
\end{equation} 
is given by a linear (distributed) filter
\begin{subequations}
\begin{equation}
\label{delayfiltera}
d\hat{x}(t|t)=\int_{t-h}^t d_sA(t,s)\hat{x}(s|t)dt +B_1udt+ X(t,t)dv
\end{equation}
\begin{equation}
\label{delayfilterb}
d_t\hat{x}(s|t)=X(s,t)dv, \; s\leq t \phantom{xxxxxxxxxxxxxxxxxxxx}
\end{equation}
\end{subequations}
where $v$ is the innovation process
\begin{equation}
\label{innovation}
dv=dy - \int_{t-h}^t d_sC(t,s)\hat{x}(s|t)dt, \quad v(0)=0, 
\end{equation}
and the gain $X$ is as defined in \cite[p.120]{lindquist1}. 
\end{thm}

For the proof of Theorem~\ref{delaythm} we shall need two lemmas. The first  is a slight reformulation of Lemma 4.1 in \cite{lindquist} and only requires that $v$ be a martingale. 

\begin{lemma}[{\rm\cite{lindquist}}]\label{decomplem}
Let $v$ be a square-integrable martingale with natural filtration
\begin{equation}
\label{Vfiltration}
{\mathcal V}_t=\sigma\{ v(s), s\inÊ[0,t]\},\quad 0\leq t\leq T
\end{equation}
and satisfying $[v_j,v_k]=\beta_j\delta_{jk}$, 
where $\beta_k$, $k=1,2,\dots,p$, are nondecreasing functions, and $\delta_{jk}$ is the Kronecker delta equal to one for $j=k$ and zero otherwise. 
With $u$ a square-integrable control process adapted to $\{{\mathcal V}_t\}$,  let 
\begin{equation}
\label{uorthdecop}
u(t)=\bar{u}(t)+\sum_{k=1}^p\int_0^tu_k(t,s)dv_k(s)+\tilde{u}(t)
\end{equation}
be the unique orthogonal decomposition for which $\bar{u}$ is deterministic and, for each $t\in [0,T]$, $\tilde{u}$ is orthogonal to the linear span of the components of $\{ v(s), s\inÊ[0,t]\}$. Moreover, let $x_0$ be a square-integrable process adapted to $\{{\mathcal V}_t\}$ and having a corresponding orthogonal decomposition
\begin{equation}
\label{x0orthdecop}
x_0(t)=\bar{x}_0(t)+\sum_{k=1}^p\int_0^tx_k^0(t,s)dv_k(s)+\tilde{x}_0(t).
\end{equation} 
Then $x=x_0 +g(u)$, defined by \eqref{eq:systemint} exchanging $z$ for $x$, has the  orthogonal decomposition
\begin{equation}
\label{xorthdecop}
x(t)=\bar{x}(t)+\sum_{k=1}^p\int_0^tx_k(t,s)dv_k(s)+\tilde{x}(t),
\end{equation}
where
\begin{subequations}
\begin{eqnarray}
\bar{x}(t)  &= & \bar{x}_0(t) +\int_0^tG(t,\tau) \bar{u}(\tau)d\tau \label{eqn(a)}\\
x_k(t,s) &= &x_k^0(t,s) +\int_s^tG(t,\tau)u_k(\tau,s)d\tau\label{eqn(b)} \\
\tilde{x}(t)&=& \tilde{x}_0(t) +\int_0^tG(t,\tau)\tilde{u}(\tau)d\tau\label{eqn(c)}
\end{eqnarray}
\end{subequations}
and
\begin{equation}
\label{Vdecomposition}
E\{V_0(x,u)\}= E\{V_0(\bar{x},\bar{u})+\sum_{k=1}^p\int_0^TV_s(x_k(\cdot,s),u_k(\cdot,s))d\beta_k + E\{V_0(\tilde{x},\tilde{u}\}).
\end{equation}
\end{lemma}

For a proof of this lemma, we refer the reader to \cite{lindquist}.

\begin{lemma}\label{innovationlem}
Let $y$ be the output process of the closed-loop system obtained after applying a deterministically well-posed feedback law $u=\pi(y)$ to the system \eqref{delaysystem}.   Then the inno\-vation process \eqref{innovation} is a Gaussian martingale, and the corresponding filtration \eqref{Vfiltration}
satisfies
\begin{equation}
\label{equivfiltrations}
{\mathcal V}_t={\mathcal Y}_t,\quad 0\leq t\leq T.
\end{equation}
\end{lemma}

\label{innovequiv}
\begin{proof}
As can be seen from the equation \eqref{integratedxdelay} and the remark following it, the process $y_0$ obtained by setting $u=0$ in \eqref{delaysystem} is given by
$dy_0=q_0(t)dt + D(t)dw$
for a  process $q_0$ adapted to  $\{{\mathcal W}_t\}$.   Define 
$dv_0=dy_0 - \hat{q}_0(t)dt,$
where $\hat{q_0}(t):=E\{ q_0(t)\mid {\mathcal Y}_t^0\}$. Now, $q_0$ and $w$ are jointly Gaussian, and therefore, for each $t\in [0,T]$,  the components of $\hat{q}_0(t)$ belong to the closed linear span of the components of the martingale $\{ y_0,\, t\in [0,T]\}$, and hence 
\begin{displaymath}
\hat{q}_0(t) = \int_0^t M(t,s)dy_0
\end{displaymath}
for some $L^2$-kernel $M$. Therefore, $v_0$ is Gaussian, and its natural filtration ${\mathcal V}_t^0$  satisfies ${\mathcal V}_t^0\subset{\mathcal Y}_t^0$. Now let $R$ be the resolvent of the Volterra equation with kernel $M$; i.e., the unique solution of the resolvent equation
\begin{displaymath}
R(t,s)=\int_s^tR(t,\tau)M(\tau,s)d\tau + M(t,s)
\end{displaymath}
\cite{Smithies,Zabreyko}. Then
\begin{displaymath}
\int_s^tR(t,s)dv_0(s)=\int_s^tM(t,s)dy_0(s)=\hat{q}_0(t),
\end{displaymath}
and hence ${\mathcal Y}_t^0\subset {\mathcal V}_t^0$. Consequently, in view of Lemma~\ref{keylemma}, 
${\mathcal V}_t^0= {\mathcal Y}_t^0={\mathcal Y}_t.$ 
Next observe that 
\begin{displaymath}
dy = q(t)+D(t)dw, \quad q(t):=q_0(t)+h(u)(t),
\end{displaymath}
where $h(u)$ is a causal (linear) function of the control $u$. Since $h(u)$ is adapted to $\{{\mathcal Y}_t\}$,
\begin{displaymath}
 \hat{q}(t):= \hat{q}_0(t)+h(u)(t),
\end{displaymath}
and therefore the innovation process \eqref{innovation} satisfies
$dv=dy- \hat{q}(t)dt=dy_0- \hat{q}_0(t)dt =dv_0.$
Equation \eqref{equivfiltrations} now follows. 

Finally, to prove that the innovation process $v$ is a martingale we need to show that 
\begin{displaymath}
E\{ v(s)-v(t)\mid {\mathcal V}_t\} =0 \; \text{\rm for all $s\geq t$}.
\end{displaymath} 
To this end, first note that 
\begin{equation}\label{martingalesum}
E\left\{ v(s)-v(t)\mid {\mathcal V}_t\right\}=E\left\{\int_t^s\tilde{q}(\tau)d\tau\mid {\mathcal V}_t\right\}+E\left\{\int_t^sB(\tau)dw\mid {\mathcal V}_t\right\},
\end{equation}
where $\tilde{q}(t):=q(t)-\hat{q}(t)$. Since all the processes are jointly Gaussian (the control-dependent terms have been canceled in forming $\tilde{q}$), independence is the same as orthogonality. Since $\tilde{q}(\tau)\perp{\mathcal V}_\tau\supset{\mathcal V}_t$ for $\tau\geq t$, the first term in \eqref{martingalesum} is zero. The second term can be written
\begin{displaymath}
E\left\{ E\left\{\int_t^sB(\tau)dw\mid {\mathcal W}_t\right\}\mid {\mathcal V}_t\right\},
\end{displaymath}
which is zero since $w$ is a martingale. 
\end{proof}

We are now in a position to prove Theorem~\ref{delaythm}.  Lemma~\ref{innovationlem} shows that the innovation process \eqref{innovation} is a martingale. It is no restriction to assume that $E\{v(t)v(t)'\}$ is diagonal; if it is not, we just normalize the  innovation process by replacing $v(t)$ by $R(t)^{-1/2}v(t)$, where $R(t):=E\{v(t)v(t)'\}>0$. Then we set $\beta_k(t):=E\{ v_k^2\}$, $k=1,2,\dots,p$. Since ${\mathcal V}_t={\mathcal Y}_t$ for $t\in [0,T]$ (Lemma~\ref{innovationlem}), admissible controls take the form \eqref{uorthdecop}. Moreover, the process $\hat{x}(t):=E\{ x(t)\mid {\mathcal Y}_t\}$ is adapted to $\{{\mathcal V}_t\}$, and hence, analogously to \eqref{uorthdecop}, it has the decomposition
\begin{equation}
\label{xhatdecomp}
\hat{x}(t)=\bar{x}(t)+\sum_{k=1}^p\int_0^tx_k(t,s)dv_k(s)+\tilde{x}(t),
\end{equation}
which now will take the place of \eqref{xorthdecop} in Lemma~\ref{decomplem}. As before, let $\hat{x}_0$ be the process $\hat{x}$ obtained by setting $u=0$. By Lemma~\ref{keylemma}, $\hat{x}_0$ does not depend on the control $u$. Moreover, since $x_0$ and $v$ are jointly Gaussian,
\begin{equation}
\label{x0hatdecomp}
\hat{x}_0(t)=\bar{x}_0(t)+\sum_{k=1}^p\int_0^tx_k^0(t,s)dv_k(s),
\end{equation}
replacing \eqref{x0orthdecop} in Lemma~\ref{decomplem}. Moreover, 
\begin{displaymath}
E\{ V_0(x,u)\}=E\{ V_0(\hat{x},u)\}+E\{ V_0(x-\hat{x},0)\},
\end{displaymath}
where the last term does not depend on the control, since $x-\hat{x}=x_0-\hat{x}_0$. Hence, by Lemma~\ref{decomplem}, the problem is now reduced to finding a control \eqref{uorthdecop} and a state process \eqref{xhatdecomp} minimizing $E\{ V_0(\hat{x},u)\}$ subject to 
\begin{subequations}
\begin{eqnarray}
\bar{x}(t)  &= & \bar{x}_0(t) +\int_0^tG(t,\tau) \bar{u}(\tau)d\tau \label{hateqn(a)}\\
x_k(t,s) &= &x_k^0(t,s) +\int_s^tG(t,\tau)u_k(\tau,s)d\tau\label{hateqn(b)} \\
\tilde{x}(t)&=&\int_0^tG(t,\tau)\tilde{u}(\tau)d\tau\label{hateqn(c)}
\end{eqnarray}
\end{subequations}
where the last equation has been modified to account for the fact that $\tilde{x}_0=0$.  Clearly, this problem decomposes into several distinct problems. First $\bar{u}$ need to chosen so that $V_0(\bar{x},\bar{u})$ is minimized subject to  \eqref{hateqn(a)}. This is a deterministic control problem with the feedback solution
\begin{equation}
\label{ubarfeedback}
\bar{u}(t)=\int_{t-h}^t d_\tau K(t,\tau)\bar{x}(\tau),
\end{equation}
where $K$ is as in \eqref{delaycontrol}. Secondly, for each $s\in [0,T]$ and $k=1,2,\dots,p$, $u_k(t,s)$ has to be chosen so as to minimize  $V_s(x_k(\cdot,s),u_k(\cdot,s))$ subject to \eqref{hateqn(b)}. This again is a deterministic control problem with the optimal feedback solution
\begin{equation}
\label{u_kfeedback}
u_k(t,s)=\int_{t-h}^t d_\tau K(t,\tau)x_k(\tau,s).
\end{equation}
Finally, $\tilde{u}$ should be chosen so as to minimize $E\{V_0(\tilde{x},\tilde{u}\})$ subject to  \eqref{hateqn(c)}. This problem clearly has the solution $\tilde{u}=0$, and hence  $\tilde{x}=0$ as well.  Combining these results inserting them into \eqref{uorthdecop} then yields the optimal feedback control 
\begin{equation*}
u(t)=\int_{t-h}^t d_\tau K(t,\tau)\big(\bar{x}(\tau)+ \sum_{k=1}^p\int_0^tx_k(t,s)dv_k(s)\big)
\end{equation*}
It remains to show that this is exactly the same as \eqref{delaystochasticcontrol}; i.e., that
\begin{equation}
\label{xhat}
\hat{x}(\tau|t)=\bar{x}(\tau)+ \sum_{k=1}^p\int_0^tx_k(t,s)dv_k(s).
\end{equation}
To this end, first note that, since the optimal control is linear in $dv$, $\hat{x}(\tau|t)$ will take the form
\begin{displaymath}
\hat{x}(\tau|t)=\bar{x}(\tau)+\int_0^t X_t(\tau,s)dv(s),
\end{displaymath}
where $\bar{x}(\tau)=E\{ x(\tau)\}$, the same as in \eqref{xhat}. Clearly $E\{[x(\tau)-\hat{x}(\tau|t)]v(s)'\}=0$ for $s\in [0,t]$, and therefore
\begin{displaymath}
E\{x(\tau)v(s)'\}=E\{\hat{x}(\tau|t)v(s)'\}=\int_0^s X_t(\tau,s)d\beta(s),
\end{displaymath}
showing that the kernel $X_t$ does not depend on $t$; hence  this index will be dropped. Now, setting $\tau=t$, comparing  with \eqref{xhatdecomp} and noting that $\tilde{x}=0$, we see that $X(t,s)$ is the matrix with columns $x_1(t,s),x_2(t,s),\dots,x_p(t,s)$, establishing \eqref{xhat}, which from now we shall write
\begin{equation}
\label{xhatmod}
\hat{x}(\tau|t)=\bar{x}(\tau)+ \int_0^tX(\tau,s)dv(s).
\end{equation}
 Hence, \eqref{delaystochasticcontrol} is the optimal control, as claimed. Moreover, 
 \begin{displaymath}
\hat{x}(\tau|t)=\hat{x}(s)+\int_s^tX(\tau,s)dv(s),
\end{displaymath}
which yields  \eqref{delayfiltera}. To derive \eqref{delayfilterb}, follow the procedure in \cite{lindquist1}.

It remains to show that the optimal control law \eqref{delaystochasticcontrol} is deterministically well-posed. To this end, it is no restriction to assume that $\bar{x}_0\equiv 0$ so that all processes have zero mean. Then it follows from \eqref{delaystochasticcontrol} and  the  unsymmetric Fubini Theorem of Cameron and Martin \cite{CM}  that 
\begin{displaymath}
u(t)=\int_0^tP(t,s)dv(s), \quad \mbox{where\;\;} P(t,s)=\int_{t-h}^t d_\tau K(t,\tau)X(\tau,s)d\tau,
\end{displaymath}
and likewise  from  \eqref{innovation} that
\begin{displaymath}
dv=dy -\int_0^tS(t,s)dv(s)dt, \quad \mbox{where\;\;}  S(t,s)=\int_{t-h}^t d_\tau C(t,\tau)X(\tau,s)d\tau.
\end{displaymath}
The function $S$ is a Volterra kernel and therefore the Volterra resolvent equation
\begin{displaymath}
V(t,s)=\int_s^tV(t,\tau)S(\tau,s)d\tau +S(t,s)
\end{displaymath}
has a unique solution $V$, from which it follows that 
\begin{displaymath}
dv=dy - \int_0^t V(t,s)dy(s).
\end{displaymath}
Then the optimal control law is given by \eqref{pioptdy}, where now $M$ is given by
\begin{displaymath}
M(t,s)=P(t,s)- \int_s^tP(t,\tau)V(\tau,s)d\tau.
\end{displaymath}
Now, for the optimal control law, $s\mapsto X(t,s)$ is of bounded variation for each $t$ \cite{lindquist1}, and hence so is $s\mapsto M(t,s)$. Hence $\pi_{\rm opt}$ can be defined samplewise as in  \eqref{piopty}.
To complete the proof that the optimal feedback loop is deterministically well-posed we proceed exactly as in the proof of Theorem~\ref{basicsepthm}, noting that in the present setting
\begin{displaymath}
\frac{\partial G}{\partial s}(t,s)=\int_s^td_\tau\begin{bmatrix}A(t.\tau)\\C(t,\tau)\end{bmatrix}\Phi(\tau,s)B_1(s),
\end{displaymath}
where $\Phi(t,s)$ is the transition matrix of $A$ \cite[p.101]{lindquist1}. 

\begin{remark}
It was shown in \cite{lindquist} that, in the case of complete state information ($y=x$), the control \eqref{delaycontrol} is optimal even when $w$ is an arbitrary (not necessarily Gaussian) martingale. 
\end{remark}

\section{Conclusions}

In studying the literature on the separation principle of stochastic control, one encounters many expositions where subtle difficulties are overlooked and inadmissible shortcuts are taken.
On the other hand, for most papers and monographs that provide rigorous derivations, one is struck by the level of mathematical sophistication and technical complexity, which makes the material hard to include in standard textbooks in a self-contained fashion.
It is our hope that our use of deterministic well-posedness provides an alternative mechanism for understanding the separation principle that is more palatable and transparent to the engineering community, while still rigorous. The new insights offered by the approach allow us to establish the separation principle also for systems driven by non-Gaussian martingale noise. However, in this more general framework the key issue of establishing well-posedness for particular control systems is challenging and more work needs to be done.

\section*{Acknowledgement}
We are indebted to an anonymous referee for significant input which has helped us improved the paper considerably.

\section{Appendix}\label{Appendix_Example}

Consider the ``uncontrolled'' observation process
$ 
dy_0=v(t)dt+\sigma dw.
$ 
If $d{\mathbb P}$ denotes the law of $(\theta,\tau,w)$ and $\Lambda(t)= e^{\sigma^{-2}\int_0^tv(s)dy_0-\frac12\sigma^{-2}\int_0^tv(s)^2ds}$, then, under a new measure
$d{\mathbb Q}:=\Lambda(T)^{-1}d{\mathbb P}$,
$y_0$ becomes a Wiener process while the law of $v$ (i.e., of $\theta$ and $\tau$) is the same as before. Under $d{\mathbb Q}$, the two processes $y_0$ and $v$ are independent. The conditional expectation is now given by (Bayes' formula \cite[page 174]{vanHandel})
\begin{eqnarray}\nonumber
E_{\mathbb P}
(v(t)|{\mathcal Y}_t) &=&
\frac{E_{{\mathbb Q}}(v(t)\Lambda(t)|{\mathcal Y}_t)}{E_{{\mathbb Q}}(\Lambda(t)|{\mathcal Y}_t)}=
 \frac{E_{{\mathbb Q}}(\theta I_{t\geq \tau}
e^{\sigma^{-2}\int_0^t\theta I_{s\geq \tau}dy_0 - \frac12 \sigma^{-2}\int_0^t I_{s\geq \tau}ds}|{\mathcal Y}_t)}
{E_{{\mathbb Q}}(
e^{\sigma^{-2}\int_0^t\theta I_{s\geq \tau}dy_0 - \frac12\sigma^{-2} \int_0^t I_{s\geq \tau}ds}|{\mathcal Y}_t)}\nonumber\\
&=& \frac{E_{{\mathbb Q}}(
I_{t\geq \tau} e^{(y_0(t)-y_0(t\wedge \tau) - \frac12(t-\tau)^+)/\sigma^2}-I_{t\geq \tau} e^{(-(y_0(t)-y_0(t\wedge \tau)) - \frac12(t-\tau)^+)/\sigma^2}|{\mathcal Y}_t)}
{E_{{\mathbb Q}}(
e^{(y_0(t)-y_0(t\wedge \tau) - \frac12(t-\tau)^+)/\sigma^2}-e^{(-(y_0(t)-y_0(t\wedge \tau)) - \frac12(t-\tau)^+)/\sigma^2}|{\mathcal Y}_t)}. \label{conditionalexp}
\end{eqnarray}
Here $t\wedge \tau:=\min(t,\tau)$, $I_{t\geq \tau}(t)=1$ when $t\geq \tau$ and $0$ otherwise, and $(t-\tau)^+=(t-\tau)I_{t\geq \tau}$. Note that $v(t)=\theta I_{t\geq \tau}(t)$. For convenience we define $\rho(t):=E_{\mathbb P}
(v(t)|{\mathcal Y}_t)$ and
\begin{eqnarray*}
\Sigma(t)&:=&\int_0^t (e^{y_0(t)-y_0(s) - \frac12(t-s))/\sigma^2}ds, \mbox{ and }\\
\bar\Sigma(t)&:=&\int_0^t  e^{(-(y_0(t)-y_0(s)) - \frac12(t-s))/\sigma^2}ds.
\end{eqnarray*}
From \eqref{conditionalexp},
$\rho(t) = N(t)/D(t)$
where
\begin{eqnarray*}
N(t) &=& \Sigma(t)-\bar\Sigma(t)\\
D(t) &=& \Sigma(t)+\bar\Sigma(t)+2(T-t).
\end{eqnarray*}
By first noting that $\Sigma$ and $\bar\Sigma$ satisfy the stochastic differential equations
\begin{eqnarray*}
d\Sigma &=& \Sigma(t) dy_0 + dt\\
d\bar\Sigma &=& -\bar\Sigma(t) dy_0 + dt,
\end{eqnarray*}
respectively, the It\^o rule applied to the expression $N(t)/D(t)$ for the conditional expectation gives the filter equations (setting $\phi=D^{-1}$)
\begin{subequations}
\begin{eqnarray}
d\rho &=& \sigma^{-2}(1-\rho(t)^2- 2(T-t)\phi(t))(dy_0-\rho dt)\\
d\phi&=& -\sigma^{-2}\phi(t)\rho(t) (dy_0(t) -\rho(t)dt).
\end{eqnarray}
\end{subequations}
Finally, noting that the innovation $dy_0-\rho dt$ is equal to $dy-\hat x dt$ for the controlled system, we obtain the filter equations \eqref{control_law}.

\spacingset{1.2}

\end{document}